\documentclass[]{amsart}

\usepackage{graphicx,color,hyperref}
\renewcommand{\labelenumi}{(\roman{enumi})}
\numberwithin{equation}{section}
\usepackage[initials,nobysame]{amsrefs}

\usepackage{mathtools}
\mathtoolsset{showonlyrefs=true}

\newtheorem{theorem}{Theorem}[section]
\newtheorem{proposition}[theorem]{Proposition}

\newtheorem{lemma}[theorem]{Lemma}
\theoremstyle{definition}

\theoremstyle{remark}
\newtheorem{remark}[theorem]{Remark}

\newcommand{\R}{\mathbf{R}}

\begin{document}

\title{Migrating elastic flows II}

\author[T.~Miura]{Tatsuya Miura}
\address[T.~Miura]{Department of Mathematics, Graduate School of Science, Kyoto University, Kitashirakawa Oiwake-cho, Sakyo-ku, Kyoto 606-8502, Japan}
\email{tatsuya.miura@math.kyoto-u.ac.jp}

\date{\today}
\keywords{Elastic flow, Huisken's problem, long-time behavior, natural boundary condition, elastica}
\subjclass[2020]{53E40 (primary), 53A04, 35B40 (secondary)}

\begin{abstract}
  We solve a variant of Huisken's problem for open curves: we construct migrating elastic flows under natural boundary conditions, extending previous work from the nonlocal flow to the purely local flow.
\end{abstract}

\maketitle


\section{Introduction}

This paper is a continuation of the author and Kemmochi's previous work \cite{KM24}, focusing on Huisken's problem for elastic flows.

The elastic flow is a widely studied example of a fourth-order parabolic geometric flow, arising as the $L^2$-gradient flow of the bending energy penalized by the length (see e.g.\ \cite{Polden1996,DziukKuwertSchaetzle2002,MantegazzaPozzetta2021,NovagaOkabe14,NovagaOkabe17,Spener2017,Lin2012,DallAcquaPozziSpener2016,Diana2024} and also the survey \cite{MPP21}).
Due to the higher-order nature of the equation, elastic flows exhibit various positivity-breaking phenomena, attributed to the lack of maximum principles.
While the question of \emph{whether} positivity breaks has been well explored for various higher-order flows \cite{GigaIto1998,GigaIto1999,MayerSimonett2000,ElliottMaier-Paape2001,Blatt2010}, the extent to \emph{how much} positivity can be broken remains largely unexamined (see \cite{miura2021optimal} for a recent development).

In this context, Huisken's problem asks \emph{whether there exists an elastic flow of closed planar curves that, while initially contained in the upper half-plane, migrates into the lower half-plane at a positive time} (\cite[Remark 1.3]{MantegazzaPozzetta2021}, \cite[p.118]{MPP21}).
This problem is profound, as it goes beyond the basic question of whether the flow simply protrudes from the upper half-plane.

Although Huisken's original problem remains unsolved, it was shown in \cite[Theorem 1.1]{KM24} that migration occurs when considering open curves with natural boundary conditions, and in addition replacing the elastic flow by its length-preserving variant. 
This result marks the first progress on the (im)possibility of migration.

However, in this scenario, the flow is not just of higher-order but also includes the nonlocal constraint of length preservation. Since nonlocality is another well-known reason for the loss of maximum principles (cf.\ \cite{Gage86, Mayer01}), the above result does not directly imply that migration is solely due to the higher-order nature.

Without the nonlocal constraint, numerical evidence in \cite[Example 3.4]{KM24} suggests that migration can still occur for open curves. Nonetheless, the theoretical analysis in this purely local case is more complex and remains an open problem.

In this paper we solve this open problem.
Let $I:=(0,1)$ and $\bar{I}:=[0,1]$.
A smooth one-parameter family of immersed curves $\gamma:\bar{I}\times[0,\infty)\to\mathbf{R}^2$ is called a \emph{(length-penalized) elastic flow} if there is a given constant $\lambda>0$ such that
\begin{align}\label{eq:flow}
  \partial_t\gamma = -2\nabla_s^2\kappa-|\kappa|^2\kappa+\lambda\kappa,
\end{align}
where $s$ denotes the arclength parameter, $\kappa=\kappa[\gamma]:=\partial_s^2\gamma$ the curvature vector ($\partial_s:=|\partial_x\gamma|^{-1}\partial_x$), and $\nabla_s\psi:=\partial_s\psi-\langle \partial_s\psi,\partial_s\gamma \rangle \partial_s\gamma$ the normal derivative along $\gamma$.
The bracket $\langle \cdot,\cdot \rangle$ denotes the Euclidean inner product.
This flow is obtained as the $L^2$-gradient flow of the modified bending energy $E_\lambda$ defined by
\[
E_\lambda[\gamma] = B[\gamma]+\lambda L[\gamma] := \int_\gamma|\kappa|^2ds + \lambda \int_\gamma ds.
\]
When specifying $\lambda$, we also call it a \emph{$\lambda$-elastic flow}.
In addition, we impose the \emph{natural boundary condition} in which the endpoints are fixed and the curvature vanishes there: for given $p_0,p_1\in\mathbf{R}^2$,
\begin{equation}\label{eq:BC}
  \gamma(0,t)=p_0,\ \gamma(1,t)=p_1,\ \kappa(0,t)=\kappa(1,t)=0 \quad \mbox{for all $t\geq0$}.
\end{equation}
Finally, define the upper and lower closed half-planes $H_\pm$, their boundary line $\Lambda_0$, and the (interior) open half-planes $H_\pm^\circ$ by 
$$H_\pm:=\{x\in\mathbf{R}^2 \mid \pm\langle x,e_2 \rangle \geq 0 \}, \quad \Lambda_0:= \partial H_\pm, \quad H_{\pm}^\circ:=H_\pm\setminus\Lambda_0.$$

Our main result then reads as follows.

\begin{theorem}\label{thm:main_migrating}
  There exists $c>0$ with the following property:
  Let $\lambda>0$ and $p_0,p_1\in \Lambda_0\subset \mathbf{R}^2$ such that $0<\lambda|p_0-p_1|^2< c$.
  Then there exists a smooth solution $\gamma:\bar{I}\times[0,\infty)\to\mathbf{R}^2$ to the $\lambda$-elastic flow \eqref{eq:flow} under the natural boundary condition \eqref{eq:BC} such that $\gamma(I\times[0,t_0])\subset H_+^\circ$ and $\gamma(I\times[t_1,\infty))\subset H_-^\circ$ hold for some $0<t_0<t_1<\infty$.
\end{theorem}

It is unclear whether the bound $c > 0$ is merely a technical condition or essential in the sense that migration cannot occur for large values of $\lambda|p_0 - p_1|^2$.
Our previous numerical simulation indicates that when $|p_0 - p_1| = 1$ and $\lambda \gg 1$, migration becomes difficult \cite[Example 3.5]{KM24}. This aligns with the fact that, in the limit as $\lambda \to \infty$, the elastic flow formally approaches the curve shortening flow (the $L^2$-gradient flow of length), which is of second order and remains in the upper half-plane due to the maximum principle.

Our proof builds on the energy-barrier method developed in the length-preserving case \cite{KM24}, in combination with a recent classification result by M\"uller and Yoshizawa for length-penalized pinned elasticae \cite{MullerYoshizawa2024}.
The primary new difficulty arises from the variational structure of stationary solutions.

In the length-preserving case, by analyzing the variational structure, we could construct a flow converging to a global minimizer of $B$ contained in $H_-$ (called the lower arc) while starting from a perturbation of a stationary solution with the second smallest energy in $H_+$ (called the upper loop).
A critical aspect of the proof involves utilizing the energy-barrier associated with zero total curvature, positioned between the two global minimizers (upper arc and lower arc).

In the length-penalized case, the straight segment is always a global minimizer of $E_\lambda$ (cf.\ Figure \ref{fig:pinned_elasticae}), acting as a trivial attractor (cf.\ \cite[Figure 5]{KM24}).
However, the asymptotic analysis of flows converging to this segment presents challenges when theoretically addressing migration, as the limit curve lies on the boundary line of the half-planes.
This prevents us from directly applying the method of flowing into a global minimizer.

To overcome this, we construct a new, more refined energy-barrier within the regime of positive total curvature. 
This again ensures the existence of a flow converging to a lower arc while starting from a perturbed upper loop.
In particular, the barrier prevents the flow from tending toward the trivial global minimizer (segment) as well as the unfavorable local minimizer (upper arc).


\begin{figure}
    \centering
    \includegraphics[width=0.25\linewidth]{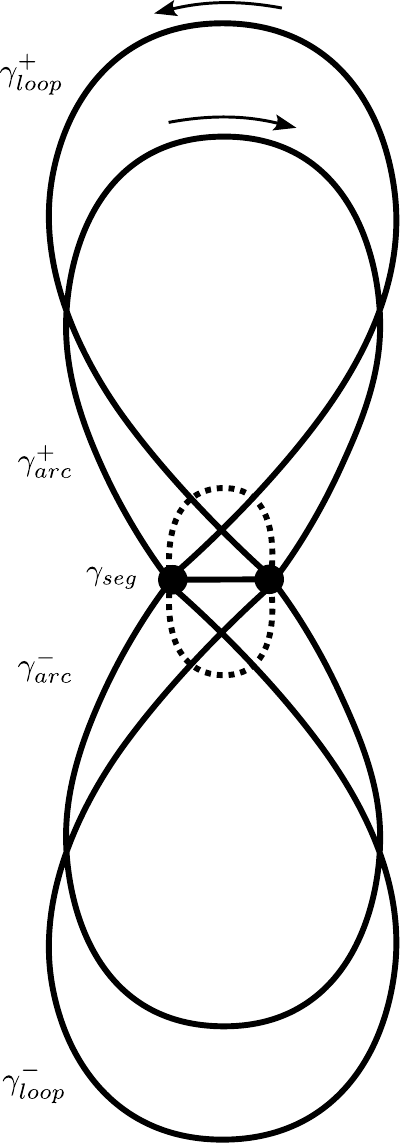}
    \caption{Low energy stationary solutions.}
    \label{fig:pinned_elasticae}
\end{figure}

\subsection*{Acknowledgments}
This work is supported by JSPS KAKENHI Grant Numbers JP21H00990, JP23H00085, and JP24K00532.

\section{Construction of migrating elastic flows}\label{sect:migrating}

In this section we prove Theorem \ref{thm:main_migrating}.
Throughout the proof, up to rescaling, we assume without loss of generality that
\[
\lambda=1,
\]
and thus consider the energy
\[
E := B+ L.
\]

\subsection{Notation}

For an immersed curve $\gamma\in H^2(I;\mathbf{R}^2)$, let $\theta=\theta[\gamma]\in H^1(I)$ denote the tangential angle function defined so that
$\partial_s\gamma = (\cos\theta,\sin\theta),$
which is unique modulo $2\pi$.
Due to the Sobolev embedding $H^2(I;\mathbf{R}^2) \hookrightarrow C^1(\bar{I};\mathbf{R}^2)$, we can assume throughout that $\gamma$ is of class $C^1$.
Let $k=k[\gamma]\in L^2(I)$ denote the signed curvature defined by
$k=\partial_s\theta,$
and let $TC$ denote the total (signed) curvature
\begin{equation}\label{eq:TC_angle}
TC[\gamma]:=\int_\gamma kds = \theta(1)-\theta(0).    
\end{equation}

For $\ell\geq0$ let $A_\ell$ be the set of immersed $H^2$-Sobolev planar curves of fixed endpoints $p_0=(0,0)$ and $p_1=(\ell,0)$,
\begin{align*}
    A_\ell := \{\gamma\in H^2_\mathrm{imm}(I;\mathbf{R}^2) \mid \gamma(0)=(0,0),\ \gamma(1)=(\ell,0) \},
\end{align*}
where
$$H^2_\mathrm{imm}(I;\mathbf{R}^2) := \big\{ \gamma\in H^2(I;\mathbf{R}^2) \mid \min_{x\in\bar{I}}|\gamma'(x)|>0  \big\}.$$

\subsection{Existence and convergence of the elastic flow}

It is known that the elastic flow \eqref{eq:flow} under the natural boundary condition \eqref{eq:BC} always has a unique global-in-time solution from any smooth initial datum \cite[Theorem 2.1]{NovagaOkabe14}, and in addition any such solution converges to a stationary solution as $t\to\infty$ (up to reparametrization) \cite[Theorem 4.2]{NovagaOkabe17}.
We summarize the properties which we will use in the following lemma.

\begin{lemma}[\cite{NovagaOkabe14,NovagaOkabe17}]\label{lem:elastic_flow}
    Let $\ell\geq0$ and $\gamma_0$ be an initial datum such that $\gamma_0\in A_\ell\cap C^\infty(I;\R^2)$ and $k[\gamma_0]=0$ on $\partial I$.
    Then the elastic flow \eqref{eq:flow} with $\lambda=1$ under the natural boundary condition \eqref{eq:BC} has a unique smooth solution $\gamma:\bar{I}\times[0,\infty)\to\R^2$.
    The flow decreases the energy $E$ in the sense that $\frac{d}{dt}E[\gamma(\cdot,t)]\leq0$ for $t\geq0$.
    In addition, as $t\to\infty$ the flow $\gamma(\cdot,t)$ smoothly converges to a limit curve $\gamma_\infty\in A_\ell$ up to reparametrization, where $\gamma_\infty$ is a smooth stationary solution such that
    \begin{align}\label{eq:pinned_elastica}
        \begin{cases}
            -2k_{ss}-k^3+k =0,\\
            k(0)=k(1)=0.
        \end{cases}
    \end{align}
\end{lemma}

\subsection{Stationary solutions}

We also summarize the known properties of solutions to \eqref{eq:pinned_elastica}, which will play an important role in our later analysis.

We first recall the characterization of the unique minimizer in the case that the endpoints coincide, $\ell=0$.

\begin{lemma}[{\cite[Proposition 5.5]{Miura_LiYau}}]\label{lem:Li-Yau}
    There exists a minimizer of $E$ in $A_0$.
    In addition, any minimizer is congruent to a unique half-fold figure-eight elastica $\gamma_*\in A_0$.
\end{lemma}

\begin{remark}
    The notion of \emph{half-fold figure-eight elastica} is defined in \cite[Definition 2.3, Remark 2.4]{Miura_LiYau}.
    Without loss of generality, hereafter we assume that the image $\gamma_*(I)$ is horizontally symmetric and contained in $H_+$, and also $\gamma_*$ has a constant-speed parametrization such that
    \[
    r_*:=TC[\gamma_*]>0.
    \]
    Under these conditions, the curve $\gamma_*$ is uniquely determined.
\end{remark}

Now, turning to the case $\ell>0$, we recall some consequences from M\"uller--Yoshizawa's classification theory for pinned penalized elasticae.
Note that in this case the minimizer of $E$ is always a trivial segment, as opposed to Lemma \ref{lem:Li-Yau}.

We say that a curve $\gamma\in A_\ell$ is \emph{strongly contained in $H_+$} (resp.\ $H_-$) if there exists $\varepsilon>0$ such that $\xi(I)\subset H_+^\circ$ (resp.\ $\xi(I)\subset H_-^\circ$) holds for any $\xi\in A_\ell$ such that $\|\xi-\gamma\|_{C^1(\bar{I};\R^2)}<\varepsilon$.

\begin{lemma}[\cite{MullerYoshizawa2024}]\label{lem:stationary}
    There exists a maximal $\ell_\dagger\in(0,\infty)$ with the following properties:
    Let $\ell\in(0,\ell_\dagger)$, and $\Gamma_\ell\subset A_\ell$ be the subset of all constant-speed curves that are solutions to \eqref{eq:pinned_elastica}.
    Then there are a unique straight segment $\gamma_{seg}^\ell\in\Gamma_\ell$, upper and lower arcs $\gamma_{arc}^{+,\ell},\gamma_{arc}^{-,\ell}\in\Gamma_\ell$, and upper and lower loops $\gamma_{loop}^{+,\ell},\gamma_{loop}^{-,\ell}\in\Gamma_\ell$ such that
    \begin{enumerate}
        \item $\gamma_{arc}^{+,\ell}$ and $\gamma_{loop}^{+,\ell}$ are strongly contained in $H_+$, while $\gamma_{arc}^{-,\ell}$ and $\gamma_{loop}^{-,\ell}$ are strongly contained in $H_-$,
        \item $TC[\gamma_{seg}^\ell]=0$, while $\pm TC[\gamma_{arc}^{\pm,\ell}]<0$ and $\pm TC[\gamma_{loop}^{\pm,\ell}]>0$,
        \item $\ell =E[\gamma_{seg}^\ell] < E[\gamma_{arc}^{+,\ell}]=E[\gamma_{arc}^{-,\ell}] < E[\gamma_{loop}^{+,\ell}]=E[\gamma_{loop}^{-,\ell}]$, and all the other (countably many) elements in $\Gamma_\ell$ have strictly larger energy $E$ than $\gamma_{loop}^{\pm,\ell}$.
    \end{enumerate}
    In addition, as $\ell\to+0$, the upper loop $\gamma_{loop}^{+,\ell}$ smoothly converges to the half-fold figure-eight elastica $\gamma_*$.
    In particular,
    \begin{equation}\label{eq:loop_convergence}
        \lim_{\ell\to+0}E[\gamma_{loop}^{+,\ell}] = E[\gamma_*], \quad \lim_{\ell\to+0}TC[\gamma_{loop}^{+,\ell}] = TC[\gamma_*].
    \end{equation}
\end{lemma}

All the important curves are plotted in Figure \ref{fig:pinned_elasticae} (where $\ell$ is dropped).

\begin{remark}
    The assertions in Lemma \ref{lem:stationary} directly or easily follow by M\"uller--Yoshizawa's classification theorem \cite[Theorem 1.1]{MullerYoshizawa2024}, combined with the explicit definitions of critical points in \cite[Definition 2.7]{MullerYoshizawa2024} and the energy comparison results in \cite[Section 4]{MullerYoshizawa2024}.
    In particular, in terms of \cite[Definition 2.7]{MullerYoshizawa2024}, our $\gamma_{arc}^{\pm,\ell}$ and $\gamma_{loop}^{\pm,\ell}$ fall into the categories of $(1,\ell,1)$-longer arcs and $(1,\ell,1)$-loops, respectively.
    In addition, the threshold distance $\ell_\dagger$ is given by $\ell_\dagger=\sqrt{\lambda_\dagger}$ with $\lambda_\dagger\simeq0.32241$ defined in \cite[Lemma 4.10]{MullerYoshizawa2024}.
    If $\ell>\ell_\dagger$, then other solutions in $\Gamma_\ell$ called shorter arcs (the dotted curves in Figure \ref{fig:pinned_elasticae}) may have less energy than the loops.
    See also \cite[Lemma 2.4 and Lemma 2.5]{KM24} for parallel properties in the fixed-length case with slightly more detailed arguments.
\end{remark}

    

\subsection{Energy-barrier argument for migration}

For $\ell\geq0$ and $r\in\R$ we define the (nonempty) subclass $A_\ell^r$ of $A_\ell$ by
\[
A_\ell^r := \{ \gamma\in A_\ell \mid TC[\gamma]=r \}.
\]

\begin{proposition}\label{prop:existence_min_TC}
    For any $\ell\geq0$ and $r\in\R$ there exists a minimizer of $E$ in $A_\ell^r$.
\end{proposition}

\begin{proof}
    This follows from a standard direct method, with the only delicate point being the need to check the non-degeneracy of length
    \begin{equation}\label{eq:nondeg_length}
        \inf_jL[\gamma_j]>0
    \end{equation}
    along a minimizing sequence $\{\gamma_j\}_j\subset A_\ell^r$.
    Otherwise, the proof proceeds in line with \cite[Theorem 3.7]{miura2024smoothcompactnesselasticae}, since the set $A_\ell^r$ is closed in the $C^1$ topology.

    The non-degeneracy \eqref{eq:nondeg_length} is trivial if $\ell>0$ since $L[\gamma_j]\geq\ell$.
    Suppose $\ell=0$.
    The Cauchy--Schwarz inequality implies
    $ 
    L[\gamma_j]B[\gamma_j] \geq (\int_{\gamma_j}|k|ds)^2.
    $
    Then by applying Fenchel's theorem to the closed curve created by concatenating $\gamma_j$ and its point-symmetry transformation, we deduce that $\int_{\gamma_j}|k|ds\geq \pi$, so that $L[\gamma_j]B[\gamma_j]\geq \pi^2$.
    This implies \eqref{eq:nondeg_length} since $\sup_{j}B[\gamma_j]\leq \sup_{j}E[\gamma_j]<\infty$.
\end{proof}

The following is a key property for our energy-barrier argument.

\begin{proposition}\label{prop:min_energy_TC}
    The minimum function $m:[0,\infty)\times\R\to[0,\infty)$ defined by 
    \[
    m(\ell,r):=\min_{\gamma\in A_\ell^r} E[\gamma]
    \]
    is continuous except at $(\ell,r)=(0,0)$.
    In addition, for any $r\in\R$,
    \[
    m(0,r) \geq m(0,r_*)=  E[\gamma_*],
    \]
    where equality holds if and only if $|r|= r_*$.
\end{proposition}

\begin{proof}
    By Proposition \ref{prop:existence_min_TC} the minimum function is well-defined: for each $(\ell,r)$ there is $\gamma^{\ell,r}\in A_\ell^r$ such that $E[\gamma^{\ell,r}]=m(\ell,r)$.
    Also by Lemma \ref{lem:Li-Yau} we deduce that $m(0,r)\geq m(0,r_*)$ with the desired equality condition.
    It remains to show the continuity of $m$.

    The upper continuity of $m$ holds at every $(\ell,r)\in[0,\infty)\times\R$.
    This follows by taking a minimizer $\gamma\in A_\ell^r$ of $E$ and, for any sequence $(\ell_j,r_j)\to(\ell,r)$, by perturbing $\gamma$ near the endpoints as in \cite[Remark 3.4]{miura2024smoothcompactnesselasticae} so that the perturbed curve $\gamma_j$ belongs to $A_{\ell_j}^{r_j}$ and converges to $\gamma$ in $H^2$.

    To prove the lower semicontinuity, we consider any sequence $(\ell_j,r_j)\to(\ell,r)\neq(0,0)$, and for each $j$ take any minimizer $\gamma_j\in A_{\ell_j}^{r_j}$ of $E$.
    By the upper semicontinuity of $m$, we have $\sup_{j}(B[\gamma_j]+L[\gamma_j])=\sup_{j}m(\ell_j,r_j)<\infty$.
    Hence, in view of the standard argument using $H^2$-weak compactness (cf.\ \cite[Lemma 3.8, or Lemma 3.2]{miura2024smoothcompactnesselasticae}), it again suffices to show the length non-degeneracy \eqref{eq:nondeg_length}.
    This is again trivial when $\ell>0$, since $L[\gamma_j]\geq\ell_j\to\ell>0$.
    If $\ell=0$, then $r\neq0$, so by the Cauchy--Schwarz inequality
    \[
    L[\gamma_j]B[\gamma_j]\geq\Big(\int_{\gamma_j}|k|ds\Big)^2\geq TC[\gamma_j]^2= r_j^2 \to r^2>0,
    \]
    from which the desired non-degeneracy follows.
    The proof is complete.
\end{proof}

We are now in a position to prove our main theorem.

\begin{proof}[Proof of Theorem \ref{thm:main_migrating}]
    We divide the argument into several steps.
    
    \emph{Step 1: Construction of an energy barrier.}
    Using the minimum function in Proposition \ref{prop:min_energy_TC}, we define
    \[
    m_*:=m(0,\tfrac{1}{2}r_*)-m(0,r_*)>0.
    \]
    By the continuity in Proposition \ref{prop:min_energy_TC},
    \[
    \lim_{(\ell,r)\to(0,\frac{1}{2}r_*)}m(\ell,r) = m(0,\tfrac{1}{2}r_*) = m(0,r_*) + m_* = E[\gamma_*]+m_*.
    \]
    Hence there is some $c_1>0$ such that
    \begin{equation}\label{eq:energy_barrier}
        m(\ell,\tfrac{1}{2}r_*)\geq E[\gamma_*]+\frac{1}{2}m_* \quad \text{for all $\ell\in[0,c_1)$}.
    \end{equation}

    \emph{Step 2: Construction of well-prepared initial data.}
    Let $\ell\in(0,\ell_\dagger)$.
    Following the same perturbation argument as in \cite[Lemma 2.13]{KM24} (based on the instability theory in \cite{MY_2024_Crelle}) we can perturb the unstable critical point $\gamma_{loop}^{+,\ell}\in A_\ell$ in Lemma \ref{lem:stationary} to construct a smooth curve
    \begin{equation}\label{eq:regularity}
        \gamma_0 \in A_\ell\cap C^\infty(\bar{I};\R^2)
    \end{equation}
    such that the following properties hold:
    \begin{align}
        &E[\gamma_0]<E[\gamma_{loop}^{+,\ell}], \label{eq:energy_decrease}\\
        &k[\gamma_0](0)=k[\gamma_0](1)=0, \label{eq:curvature_zero}
    \end{align}
    and $\gamma_0$ is arbitrarily close to $\gamma_{loop}^{+,\ell}$ in $H^2(I;\R^2)$.
    Due to this closeness we may particularly assume that
    \begin{align}
        TC[\gamma_0] \geq TC[\gamma_{loop}^{+,\ell}]-\frac{1}{4}r_*, \label{eq:TC_close}
    \end{align}
    and also, combined with Lemma \ref{lem:stationary} (i),
    \begin{align}
        \text{$\gamma_0\in A_\ell$ is strongly contained in $H_+$}. \label{eq:strongly_contained}
    \end{align}
    (In fact, the construction here may be much easier than \cite{KM24} since we do not need to prescribe the length.)
    In particular, thanks to \eqref{eq:regularity} and \eqref{eq:curvature_zero} we can choose $\gamma_0\in A_\ell$ as an initial datum in Lemma \ref{lem:elastic_flow}.

    In addition, we obtain more properties of the above $\gamma_0\in A_\ell$ for smaller $\ell$.
    By the convergence properties in \eqref{eq:loop_convergence}, there is $c_2\in(0,\ell_\dagger]$ such that for any $\ell\in(0,c_2)$,
    \[
    \Big| E[\gamma_{loop}^{+,\ell}]-E[\gamma_*] \Big|\leq \frac{1}{2}m_*, \quad \Big| TC[\gamma_{loop}^{+,\ell}]-TC[\gamma_*] \Big| \leq \frac{1}{4}r_*.
    \]
    Combining these with \eqref{eq:energy_decrease} and \eqref{eq:TC_close}, we deduce that
    \begin{align}
        &E[\gamma_0]< E[\gamma_*]+\frac{1}{2}m_*,\label{eq:initial_energy}\\
        &TC[\gamma_0] \geq TC[\gamma_*]-\frac{1}{2}r_* = \frac{1}{2}r_*. \label{eq:initial_TC}
    \end{align}

    \emph{Step 3: Proof of migration.}
    Let $c:=\min\{c_1,c_2\}$, and let $\ell\in(0,c)$.
    We prove that the (smooth) elastic flow $\gamma$ starting from $\gamma_0\in A_\ell$ satisfies the desired properties.
    
    By \eqref{eq:strongly_contained} and the smooth convergence of $\gamma$ as $t\to+0$ (well-posedness in Lemma \ref{lem:elastic_flow}), there is $t_0>0$ such that $\gamma(I\times[0,t_0])\subset H_+^\circ$.

    By \eqref{eq:energy_decrease} and by the decreasing property of $E$, the asymptotic limit $\gamma_\infty\in A_\ell$ in Lemma \ref{lem:elastic_flow} must satisfy $E[\gamma_\infty]<E[\gamma_{loop}^{+,\ell}]$.
    Then Lemma \ref{lem:stationary} (iii) ensures that, up to reparametrization,
    \begin{equation}\label{eq:limit_candidate}
        \gamma_\infty \in \{ \gamma_{seg}^\ell,\gamma_{arc}^{+,\ell},\gamma_{arc}^{-,\ell}\}.
    \end{equation}
    In addition, we also need to have 
    \begin{equation}\label{eq:limit_TC_positive}
        TC[\gamma_\infty]>0.
    \end{equation}
    Indeed, if otherwise $TC[\gamma_\infty]\leq0$, then by \eqref{eq:initial_TC} and the mean value theorem there is some $t_*\geq 0$ such that $TC[\gamma(\cdot,t_*)]=\frac{1}{2}r_*$, but in this case, estimate \eqref{eq:energy_barrier} implies that $E[\gamma(\cdot,t_*)]\geq E[\gamma_*]+\frac{1}{2}r_*$, which contradicts \eqref{eq:initial_energy} and the decreasing property of $E$.

    Combining \eqref{eq:limit_candidate} and \eqref{eq:limit_TC_positive} with Lemma \ref{lem:stationary} (ii), we deduce that $\gamma_\infty=\gamma_{arc}^{-,\ell}$.
    By Lemma \ref{lem:stationary} (i) and the smooth convergence $\gamma(\cdot,t)\to \gamma_\infty$ as $t\to\infty$ (up to reparametrization), we further deduce that there is a $t_1\in(t_0,\infty)$ such that $\gamma(I\times[t_1,\infty))\subset H_-^\circ$.
    The proof is now complete.
\end{proof}

\begin{remark}
    Our approach to the proof of Theorem \ref{thm:main_migrating} can even provide an explicit upper bound for allowable $\ell>0$.
    The key idea in our proof is to use the fact that there is some $r\in(0,r_*)$ such that, starting from the perturbed upper loop, the elastic flow cannot go beyond the energy-barrier $m(\ell,r)$.
    Consequently, the same reasoning applies to all $\ell\in(0,\ell_\dagger)$ for which
    \[
    \max_{r\in[0,r_*]}m(\ell,r)\geq E[\gamma_{loop}^{+,\ell}].
    \]
    However, this bound is likely not optimal for migration, so we do not attempt to compute it here.
\end{remark}

\begin{remark}
    The flow constructed here does not converge to the segment, thus different from the numerical result in \cite[Figure 5]{KM24} (but rather similar to the length-preserving case \cite[Figure 2]{KM24}).
    It is theoretically open whether there exists a migrating elastic flow that converges to the segment.
\end{remark}

\bibliography{ref_KM}

\end{document}